\documentclass[12pt]{article}
\usepackage[]{amsmath,amssymb}
\usepackage[]{amsthm,enumerate}
\usepackage{url}
\usepackage[T1]{fontenc}
\usepackage[]{comment}
\usepackage{lscape}
\usepackage{cite}

\newtheorem{theorem}{Theorem}
\newtheorem{proposition}[theorem]{Proposition}
\newtheorem{lemma}[theorem]{Lemma}
\newtheorem{corollary}[theorem]{Corollary}
\theoremstyle{definition}

\theoremstyle{remark}

\newcommand{\ZZ}{\mathbb{Z}}

\begin{document}

\title{
Supplementary difference sets related to a certain class of
complex spherical 2-codes 
}

\date{\today}

\author{
Makoto Araya\thanks{Department of Computer Science,
Shizuoka University,
Hamamatsu 432--8011, Japan.
email: araya@inf.shizuoka.ac.jp},
Masaaki Harada\thanks{
Research Center for Pure and Applied Mathematics,
Graduate School of Information Sciences,
Tohoku University, Sendai 980--8579, Japan.
email: mharada@m.tohoku.ac.jp.}
and
Sho Suda\thanks{Department of Mathematics Education, Aichi University of Education, Kariya 448--8542, Japan.
email: suda@auecc.aichi-edu.ac.jp}
}

\maketitle


\begin{center}
{\bf Dedicated to Professor Hiroshi Kimura on His 80th Birthday} 
\end{center}

\begin{abstract}
In this paper, we study skew-symmetric $2$-$\{v;r,k;\lambda\}$
supplementary difference sets 
related to a certain class of complex spherical 2-codes.
A classification of such supplementary difference sets
is complete for $v \le 51$.
\end{abstract}

\section{Introduction}\label{sec:Intro}

Let $\Omega(d)$ denote the complex unit
sphere in $\mathbb{C}^d$. 
For a finite set $X$ in $\Omega(d)$, define
\[
A(X)=\{x^*y \mid  x,y \in X, x \ne y \}, 
\]
where $x^*$ is the transpose conjugate of a column vector $x$. 
A finite set $X$ is called a {\em complex spherical $2$-code} 
if $|A(X)|=2$ and $A(X)$ contains an imaginary number. 
%
%
A complex spherical $2$-code $X$ with $A(X)=\{\alpha,\bar{\alpha}\}$ has
the structure of a tournament $(X,E)$, where $E=\{(x,y)\in X\times X
\mid x^*y=\alpha \}$~\cite{NS}. 
We say that the tournament $(X,E)$
is attached to the complex spherical $2$-code $X$. 
\begin{theorem}[Nozaki and Suda~{\cite[Theorem 4.8]{NS}}]\label{thm:i1}
\label{thm:bound2code}
Let $X$ be a complex spherical $2$-code in $\Omega(d)$. 
Let $A$ be the adjacency matrix of the tournament $G$ attached to $X$.
\begin{enumerate}
\item\label{enu:1} $|X|\leq 2d+1$ if $d$ is odd, and $|X|\leq 2d$ if $d$ is even. 
\item 
$|X| = 2d+1$ for odd $d$ if and only if $G$ is a doubly regular tournament. 
\item 
$|X|= 2d$  for even $d$ if and only if $I+A-A^T$ is a skew-Hadamard matrix,
where $I$ is the identity matrix
and $A^T$ denotes the transposed matrix of $A$.
\item $|X|=2d$  for odd $d$  if and only if one of the following occurs:
\begin{enumerate}[\rm (a)]
\item 
$A$ is obtained as the adjacency matrix of the induced subgraph 
of some doubly regular tournament by deleting a certain vertex.  
\item There exists a permutation matrix $P$ such that 
\begin{align*}
P(I+A-A^T)(I+A-A^T)^T P^T=
\begin{pmatrix}
\alpha I+\beta J& O\\
O&\alpha I+\beta J
\end{pmatrix},
\end{align*} 
for some integers $\alpha,\beta$ with $\alpha \ge 2, \beta \ge 1$,
where 
$J$ denotes the all-one matrix and 
$O$ denotes the zero matrix of appropriate size.
\end{enumerate}
\end{enumerate}
\end{theorem}

Doubly regular tournaments have been widely studied 
(see e.g., \cite{KS94, NSt, RB72, R86, S69}).
Skew-Hadamard matrices are a class of Hadamard matrices, which
has been widely studied (see e.g., 
\cite{BS, FKS, GS, NSt, RB72, S69, W71}).
These motivate our investigation of matrices $M$ 
satisfying the following conditions:
\begin{align}
\label{eq:C1}
&\text{$M$ is a $2d \times 2d$ $(1,-1)$-matrix with $d$ odd},\\
\label{eq:C2}
&\text{$M-I=-(M-I)^T$, that is, $M$ is skew-symmetric},\\
\label{eq:C3}
&M M^T=
\left(
\begin{smallmatrix}
\alpha I+\beta J& O\\
O&\alpha I+\beta J
\end{smallmatrix}
\right)
\text{for some integers $\alpha,\beta$
with $\alpha \ge 2,\beta \ge 1$.}
\end{align}
In this paper, with this motivation, we study
skew-symmetry for $2$-$\{v;r,k;\lambda\}$
supplementary difference sets satisfying the following conditions:
\begin{align}
\label{eq:CS1}
&\text{$v$ is an odd positive integer},\\
\label{eq:CS2}
&\text{$4(r+k-\lambda)\ge 2$}, \\
\label{eq:CS3}
&\text{$2(v-2(r+k-\lambda))\ge 1$}.
\end{align}
These supplementary difference sets
give matrices $M$ satisfying \eqref{eq:C1}--\eqref{eq:C3} (Proposition~\ref{prop:SDS}).

This paper is organized as follows.
In Section~\ref{sec:Pre}, we give definitions and we recall notions on
supplementary difference sets and $D$-optimal designs.
Some basic facts on these subjects 
are also provided.
In Section~\ref{sec:skewSDS}, we give some observations
on skew-symmetric supplementary difference sets.
In Section~\ref{sec:method}, we describe how to classify
skew-symmetric supplementary difference sets 
satisfying \eqref{eq:CS1}--\eqref{eq:CS3}.
In Section~\ref{sec:SDS},
we give a classification of skew-symmetric  $2$-$\{v;r,k;\lambda\}$
supplementary differences
sets satisfying \eqref{eq:CS1}--\eqref{eq:CS3}
for $v \le 51$ (Theorem~\ref{thm:SDS-Classification}).
This is the main result of this paper.
Skew-symmetric circulant $D$-optimal designs 
satisfying \eqref{eq:Gram} are corresponding to a special class of 
supplementary difference sets.
In Section~\ref{sec:ccsdd}, 
as a consequence of Theorem~\ref{thm:SDS-Classification},
we give a 
classification of skew-symmetric circulant $D$-optimal
designs meeting \eqref{eq:det} for orders up to $110$. 

\section{Preliminaries}\label{sec:Pre}

In this section, we give definitions and we recall notions on
supplementary difference sets and $D$-optimal designs.
Some basic facts on these subjects
are also provided.

\subsection{Supplementary difference sets}\label{sec:SDS-D}

Let $\ZZ_{v} = \{0,1,\ldots,v-1\}$ be the ring of integers modulo $v$, where $v >2$.
%
For $A \subset \ZZ_{v}$ and $i \in \ZZ_{v}$, define
\begin{multline*}
P_{A}(i) = |\{(x,y) \in A \times A \mid y-x = i\}| \text{ and } \\
P_{A} = (P_{A}(1),P_{A}(2),\ldots,P_{A}(v-1)).
\end{multline*}
Let $A$ and $B$ be an $r$-subset and a
$k$-subset of $\ZZ_v$, respectively.
If a pair $(A,B)$ satisfies 
$$P_{A}+P_{B}=(\lambda,\lambda,\ldots,\lambda),$$
then it is called a {\em $2$-$\{v;r,k;\lambda\}$ supplementary difference
set}.
We refer to~\cite{CK,KKS,Seberry72,Seberry74}
for basic facts on supplementary difference sets.

\begin{lemma}[{Wallis~\cite[Lemma~1]{Seberry72}}]
If there exists a $2$-$\{v;r,k;\lambda\}$ supplementary difference set,
then 
\begin{align}\label{eq:SDS}
r(r-1)+k(k-1)=\lambda(v-1).
\end{align}
\end{lemma}


Chadjipantelis and Kounias~\cite[Appendix]{CK} gave a
correspondence between 
$2$-$\{v;r,k;\lambda\}$ supplementary difference sets
and pairs of circulant matrices.
Let $A$ and $B$ be an $r$-subset and a
$k$-subset of $\ZZ_v$, respectively.
Let $R_1$ and $R_2$ be the circulant $v \times v$ $(1,-1)$-matrices with
first rows $r_1=(r_{1,1},r_{1,2},\ldots,r_{1,v})$ and 
$r_2=(r_{2,1},r_{2,2},\ldots,r_{2,v})$, respectively.
The correspondence was defined as follows:
$r_{1,i+1}=-1$ if $i \in A$,
$r_{1,i+1}=1$ if $i \not\in A$
and 
$r_{2,i+1}=-1$ if $i \in B$,
$r_{2,i+1}=1$ if $i \not\in B$.

\begin{lemma}[{Chadjipantelis and Kounias~\cite[Appendix]{CK}}]\label{lem:SDS}
A pair $(A,B)$ is a $2$-$\{v;r,k;\lambda\}$ supplementary difference set
if and only if 
$R_1R_1^T+R_2R_2^T=4(r+k-\lambda)I+2(v-2(r+k-\lambda))J$.
\end{lemma}

\subsection{$D$-optimal designs and supplementary difference sets}
\label{sec:dd}

A {\em $D$-optimal design} of order $n$ is an $n \times n$ $(1,-1)$-matrix
having maximum determinant.
Ehlich~\cite{Ehlich} showed that for $n \equiv 2 \pmod 4$ and
$n> 2$, any $n \times n$ $(1,-1)$-matrix $M$ satisfies
\begin{equation}
\label{eq:det}
\det M \le (2n-2)(n-2)^{(n-2)/2},
\end{equation}
and
that equality is possible only if $2n-2$ is a sum of two
perfect squares.
Moreover, if $n=2v \equiv 2 \pmod 4$, and
both $R_1$ and $R_2$ are $v \times v$ commutative $(1,-1)$-matrices such that
\begin{equation}\label{eq:Gram}
R_1 R_1^T + R_2 R_2^T = (2v-2)I + 2 J,
\end{equation}
then
\begin{equation}\label{eq:D-opt}
X(R_1,R_2)=
\left(\begin{array}{rr}
R_1 & R_2 \\
-R_2^T & R_1^T
\end{array}\right)
\end{equation}
is a $D$-optimal design meeting the above bound \eqref{eq:det}~\cite{Ehlich}.
%

If the matrices $R_1$ and $R_2$ in \eqref{eq:D-opt} are circulant,
then $X(R_1,R_2)$ is called a {\em circulant $D$-optimal design} 
meeting \eqref{eq:det}~\cite{KKNK}.
Most of the known $D$-optimal designs meeting \eqref{eq:det} are circulant
(see e.g., \cite{AH, CK, FKS, KKS, KKNK}).
If $X(R_1,R_2)$ is a circulant $D$-optimal design meeting \eqref{eq:det},
then it was shown in~\cite{CK} that
\begin{align}\label{eq:1}
(v-2r)^2+(v-2k)^2=2n-2,
\end{align}
where
$r$ and $k$ are the numbers of $-1$'s in the first rows of
$R_1$ and $R_2$, respectively.

By Lemma~\ref{lem:SDS}, we have the following:

\begin{lemma}[{Chadjipantelis and Kounias~\cite[Appendix]{CK}}]
\label{lem:SDS-Doptimal}
Let $A$ and $B$ be an $r$-subset and a
$k$-subset of $\ZZ_v$, respectively.
Let $R_1$ and $R_2$ be the corresponding circulant 
$v \times v$ $(1,-1)$-matrices described in Section~\ref{sec:SDS-D}.
A pair $(A,B)$ is a $2$-$\{v;r,k;r+k-(v-1)/2\}$ supplementary difference set
if and only if 
$X(R_1,R_2)$ in \eqref{eq:D-opt} is a circulant $D$-optimal design 
of order $2v$ meeting \eqref{eq:det}, where
$r$ and $k$ are the numbers of $-1$'s in the first rows of
$R_1$ and $R_2$, respectively.
\end{lemma}

\section{Skew-symmetric supplementary difference sets}
\label{sec:skewSDS}

Let $(A,B)$ be  a supplementary difference set.
Let $R_1$ and $R_2$ denote the corresponding 
circulant $v \times v$ $(1,-1)$-matrices 
described in Section~\ref{sec:SDS-D}.
Then we consider the following matrix:
\begin{equation}\label{eq:skewSDS}
X(R_1,R_2)=
\left(\begin{array}{rr}
R_1 & R_2 \\
-R_2^T & R_1^T
\end{array}\right).
\end{equation}
We call $(A,B)$ {\em skew-symmetric} if 
the corresponding matrix 
$X(R_1,R_2)$  in \eqref{eq:skewSDS} is skew-symmetric. 
Equivalently, $(A,B)$ is skew-symmetric
if $A$ satisfies the condition that $0 \not\in A$ and if $i \in A$ then $-i \not\in A$.
In~\cite{BS, S69,W71}, skew-symmetric
$2$-$\{v;(v-1)/2,(v-1)/2;(v-3)/2\}$ supplementary difference sets 
are called complementary difference sets and these 
difference sets were used to construct skew-Hadamard
matrices.

\begin{lemma}\label{lem:skew}
The matrix
$X(R_1,R_2)$  in \eqref{eq:skewSDS}
is skew-symmetric if and only if $r_{1,1}=1$ and $r_{1,i}=-r_{1,v+2-i}$ 
$(i=2,3,\ldots,v)$. 
If $X(R_1,R_2)$ is skew-symmetric, then $r=\frac{v-1}{2}$. 
\end{lemma}
\begin{proof}
The elementary proof is omitted.
\end{proof}

By Lemma~\ref{lem:SDS}, we have the following:

\begin{proposition}\label{prop:SDS}
If there exists a  skew-symmetric
$2$-$\{v;r,k;\lambda\}$ supplementary difference 
set satisfying  \eqref{eq:CS1}--\eqref{eq:CS3}.
Then there exists a matrix $M$ satisfying \eqref{eq:C1}--\eqref{eq:C3} 
for $(\alpha,\beta)=(4(r+k-\lambda),2(v-2(r+k-\lambda)))$.
\end{proposition}

Now we give a remark on the condition \eqref{eq:CS2}
for skew-symmetric
$2$-$\{v;r,k;\lambda\}$ supplementary difference sets.

\begin{proposition}\label{prop:CS2}
Suppose that $k \le \frac{v-1}{2}$.
If there exists a  skew-symmetric
$2$-$\{v;r,k;\lambda\}$ supplementary difference set $(A,B)$,
then $r+k-\lambda\ge 1$, that is,
$(A,B)$ satisfies \eqref{eq:CS2}.
\end{proposition}
\begin{proof}
By Lemma~\ref{lem:skew}, $r=\frac{v-1}{2}$.
Hence, it follows from \eqref{eq:SDS} that
$r+k-\lambda=\frac{(v+2k)(v-2k)-1+4kv}{4(v-1)}$.
From the assumption, $r+k-\lambda >0$.
The result follows.
\end{proof}

For the case $k\in\{0,1\}$, $2$-$\{v;r,k;\lambda\}$ supplementary difference sets are characterized as follows.
Although the following characterization is somewhat trivial,
it was not explicitly stated in the literature.
We give a proof for the sake of completeness.

\begin{proposition}
The following statements are equivalent.  
\begin{enumerate}
\item There exists 
a  skew-symmetric $2$-$\{4m-1;2m-1,k;m-1\}$ supplementary difference set 
with $k=0$ and $1$. 
\item There exists a circulant Hadamard $2$-$(4m-1,2m-1,m-1)$ design with incidence matrix $M$ satisfying that $M+M^T+I=J$.
\end{enumerate}
\end{proposition}
\begin{proof}
Suppose that there exists
a skew-symmetric $2$-$\{4m-1;2m-1,k;m-1\}$ supplementary difference set $(A,B)$ 
with $k \in \{0,1\}$.
Then $A$ is a $(4m-1,2m-1,m-1)$-difference set. 
Let $M$ be an incidence matrix of $A$. 
Then $M$ is an incidence matrix of
a circulant Hadamard $2$-$(4m-1,2m-1,m-1)$ design.
Since $A$ satisfies the condition that if $i \in A$ then $-i \not\in A$,
$M$ satisfies the condition that $M+M^T+I=J$. 

Suppose that there exists a circulant Hadamard $2$-$(4m-1,2m-1,m-1)$ 
design with incidence matrix $M$ satisfying that $M+M^T+I=J$.
By reversing the above argument, 
a $(4m-1,2m-1,m-1)$-difference set $A$ satisfying
the condition that if $i \in A$ then $-i \not\in A$
is constructed.
Then $(A,\emptyset)$ and $(A,\{0\})$ are
skew-symmetric  $2$-$\{v;r,k;\lambda\}$
supplementary difference sets 
with parameters 
$(v,r,k,\lambda)=(4m-1,2m-1,0,m-1)$ and 
$(4m-1,2m-1,1,m-1)$, respectively.
\end{proof}

Suppose that $p$ is a prime with $p\equiv 3\pmod{4}$.
Then it is well known that there exists a circulant Hadamard 
$2$-$(p,\frac{p-1}{2},\frac{p-3}{4})$ design with incidence matrix $A$ satisfying that 
$A+A^T+I=J$ (see e.g., \cite[Lemma~7.10]{HSS-OA}).
This implies the existence of skew-symmetric supplementary difference sets with
parameters
$2$-$\{p;\frac{p-1}{2},0;\frac{p-3}{4}\}$ and 
$2$-$\{p;\frac{p-1}{2},1;\frac{p-3}{4}\}$.



\section{Classification method}\label{sec:method}

In this section, we describe how to classify
skew-symmetric supplementary difference sets
satisfying \eqref{eq:CS1}--\eqref{eq:CS3}.

\subsection{Equivalent supplementary difference sets}
\label{subsec:Eq}

If $(A,B)$ is a supplementary difference set, then
the following pairs
\begin{itemize}
\item[(E0)] $(\ZZ_{v}\backslash A,B)$ and $(A,\ZZ_{v}\backslash B)$,
\item[(E1)] $(B,A)$,
\item[(E2)] $(\pm A+a, \pm B+b)$ for any $a, b \in \ZZ_v$, 
\item[(E3)] $(dA, dB)$ for any $d \in U(\ZZ_v)$
\end{itemize}
are also supplementary difference sets,
where $U(\ZZ_v)=\{d \in \{1,2,\ldots,v-1\} \mid \gcd(d,v)=1\}$
and $d$ is regarded as an integer for $\gcd(d,v)=1$.
These supplementary difference sets are called
{\em equivalent}~\cite{KKNK}.

\subsection{Classification method}

Let $(A,B)$ be a skew-symmetric $2$-$\{v;r,k;\lambda\}$ 
supplementary difference set satisfying \eqref{eq:CS1}, \eqref{eq:CS3}.
By Lemma~\ref{lem:skew}, $r=\frac{v-1}{2}$.
By (E0), 
we may assume without loss of generality that 
$k \le \frac{v-1}{2}$.
We note that $(A,B)$ satisfies \eqref{eq:CS2} by Proposition~\ref{prop:CS2} 
under this assumption.
In addition, if $r=k$, then
it follows from \eqref{eq:SDS} that
$2(v-2(r+k-\lambda)) =-2$.
Hence, 
we may assume without loss of generality that 
\begin{equation}\label{eq:New}
k < \frac{v-1}{2}=r.  
\end{equation}
Since $A$ corresponds to a skew-symmetric matrix, 
there exists $A' \subset \ZZ_v'$ such that $A=A'\cup \{v-j \mid j 
\in \ZZ_v' \setminus A'\}$,
where $\ZZ_v' = \{1,2,\dots,(v-1)/2\}\subset \ZZ_v$.
By (E2), $(A,B+b)$ is a skew-symmetric supplementary difference set for any $b \in \ZZ_v$.
We classify skew-symmetric  $2$-$\{v;(v-1)/2,k;\lambda\}$ 
supplementary difference sets 
satisfying \eqref{eq:CS1}, \eqref{eq:CS3} by the following steps.

\begin{enumerate}[\rm (i)]
\item We calculate $\overline{\cal A}=
\{A'\cup \{v-j \mid j \in \ZZ_v' \setminus A'\} \mid A' \subset \ZZ_v'\}$.
Then we find ${\cal A} = \{A \in \overline{\cal A} \mid 
P_A(i) \le \lambda \text{ for all }i\}$.

\item We calculate 
$\overline{\cal B}=\{B \subset \ZZ_v \mid |B|=k, B \preceq B+b \text{ for any } b
\in \ZZ_v\}$,
where $\preceq$ is  a natural lexicographic order 
on $k$-subsets of $\ZZ_v$.
Then we find ${\cal B} = \{B \in \overline{\cal B} \mid 
P_B(i) \le \lambda \text{ for all }i\}$.

\item We construct ${\cal AB}=\{(A,B) \in {\cal A}\times{\cal B} \mid
      P_A+P_B=(\lambda,\lambda,\dots,\lambda)\}$.

\item 
We classify ${\cal AB}$.
\end{enumerate}

In Step (i) (resp.\ (ii)), we found all $((v-1)/2)$-subsets 
(resp.\ $k$-subsets) of $\ZZ_v$ by a computer program implemented 
in {\tt C} language
using functions from the GNU Scientific Library (GSL) software library,
then we output $A$ and $(\lambda,\lambda,\dots,\lambda)-P_A$
(resp.\ $B$ and $P_B$) to a file.
We sorted the above data by $(\lambda,\lambda,\dots,\lambda)-P_A$ 
(resp.\ $P_B$).
We found a pair $(A,B)$ with $(\lambda,\lambda,\dots,\lambda)-P_A=P_B$ 
in Step (iii).
Two skew-symmetric  $2$-$\{v;(v-1)/2,k;\lambda\}$ 
supplementary difference sets $(A,B)$ and $(A',B')$ are 
equivalent if and only if $(A',B')$ is an element of 
$\{(\pm dA+a,\pm dB+b) \mid d \in U(\ZZ_v), a,b \in \ZZ_v \}$.
In Step (iv), for $(A,B)$ and $(A',B')$, we determined whether there exist
$d \in U(\ZZ_v)$ and $a,b \in \ZZ_v$ such that
$(A',B')=(dA+a,dB+b),(dA+a,-dB+b),(-dA+a,dB+b)$ or 
$(-dA+a,-dB+b)$.
This was done 
by using the program implemented in {\tt C} language.

\section{Classification of skew-symmetric supplementary difference sets}
\label{sec:SDS}

In this section,
we give a classification of skew-symmetric
$2$-$\{v;r,k;\lambda\}$ supplementary differences
sets satisfying \eqref{eq:CS1}--\eqref{eq:CS3} for $v \le 51$.
This is the main result of this paper.
As described in Proposition~\ref{prop:SDS}, 
a  skew-symmetric
$2$-$\{v;r,k;\lambda\}$ supplementary difference set
satisfying \eqref{eq:CS1}--\eqref{eq:CS3},
gives a matrix $M$ satisfying \eqref{eq:C1}--\eqref{eq:C3} 
for $(\alpha,\beta)=(4(r+k-\lambda),2(v-2(r+k-\lambda)))$.

We call $(v,r,k,\lambda)$ {\em feasible parameters} for 
supplementary difference sets if
$(v,r,k,\lambda)$ satisfies \eqref{eq:CS1}--\eqref{eq:CS3},
\eqref{eq:SDS} and \eqref{eq:New}
(see Proposition~\ref{prop:CS2} for \eqref{eq:CS2}).
In Table~\ref{Tab:Par},
we list the feasible parameters $(v,r,k,\lambda)$ for $v\leq 75$.

\begin{table}[thbp]
\caption{Parameters of skew-symmetric supplementary difference sets}
\label{Tab:Par}
\begin{center}
{\small
\begin{tabular}{l|c|l|c}
\noalign{\hrule height0.8pt}
\multicolumn{1}{c|}{$(v,r,k,\lambda)$}
& $N(v,r,k,\lambda)$ &
\multicolumn{1}{c|}{$(v,r,k,\lambda)$}
& $N(v,r,k,\lambda)$ \\
\hline
$( 3,  1,  0,  0)$ & 1 &$(43, 21, 15, 15)$ & 0\\
$( 7,  3,  0,  1)$ & 1 &$(45, 22, 11, 13)$ & 0\\
$( 7,  3,  1,  1)$ & 1 &$(47, 23,  0, 11)$ & 1\\
$(11,  5,  0,  2)$ & 1 &$(47, 23,  1, 11)$ & 1\\
$(11,  5,  1,  2)$ & 1 &$(49, 24,  9, 13)$ & 0\\
$(13,  6,  3,  3)$ & 1 &$(51, 25,  0, 12)$ & 0\\
$(15,  7,  0,  3)$ & 0 &$(51, 25,  1, 12)$ & 0\\
$(15,  7,  1,  3)$ & 0 &$(53, 26, 14, 16)$ & ?\\
$(19,  9,  0,  4)$ & 1 &$(55, 27,  0, 13)$ & 0\\
$(19,  9,  1,  4)$ & 1 &$(55, 27,  1, 13)$ & 0\\
$(21, 10,  6,  6)$ & 1 &$(57, 28, 21, 21)$ & ?\\
$(23, 11,  0,  5)$ & 1 &$(59, 29,  0, 14)$ & 1\\
$(23, 11,  1,  5)$ & 1 &$(59, 29,  1, 14)$ & 1\\
$(25, 12,  4,  6)$ & 0 &$(61, 30,  6, 15)$ & 0\\
$(27, 13,  0,  6)$ & 0 &$(61, 30, 10, 16)$ & 0\\
$(27, 13,  1,  6)$ & 0 &$(61, 30, 15, 18)$ & ?\\
$(29, 14,  7,  8)$ & 1 &$(63, 31,  0, 15)$ & 0\\
$(31, 15,  0,  7)$ & 1 &$(63, 31,  1, 15)$ & 0\\
$(31, 15,  1,  7)$ & 1 &$(67, 33,  0, 16)$ & $1$\\
$(31, 15,  6,  8)$ & 1 &$(67, 33,  1, 16)$ & $1$\\
$(31, 15, 10, 10)$ & 1 &$(67, 33, 12, 18)$ & ?\\
$(35, 17,  0,  8)$ & 0 &$(67, 33, 22, 23)$ & ?\\
$(35, 17 , 1,  8)$ & 0 &$(69, 34, 18, 21)$ & ?\\
$(37, 18, 10, 11)$ & 0 &$(71, 35,  0, 17)$ & $1$\\
$(39, 19,  0,  9)$ & 0 &$(71, 35,  1, 17)$ & $1$\\
$(39, 19,  1,  9)$ & 0 &$(71, 35, 15, 20)$ & ?\\
$(41, 20,  5, 10)$ & 0 &$(71, 35, 21, 23)$ & ?\\
$(43, 21,  0, 10)$ & 1 &$(73, 36, 28, 28)$ & ?\\
$(43, 21,  1, 10)$ & 1 &$(75, 37,  0, 18)$ & 0\\
$(43, 21,  7, 11)$ & 0 &$(75, 37,  1, 18)$ & 0\\
\noalign{\hrule height0.8pt}
\end{tabular}
}
\end{center}
\end{table}

By an approach given in Section~\ref{sec:method},
our exhaustive computer search completed
a classification of skew-symmetric  $2$-$\{v;r,k;\lambda\}$
supplementary difference sets 
satisfying \eqref{eq:CS1}--\eqref{eq:CS3} for 
the feasible parameters in Table~\ref{Tab:Par}
with $v \le 51$.
We used a computer with CPU Intel(R) Core(TM) i7 4790k, 4 Core.

\begin{theorem}\label{thm:SDS-Classification}
Suppose that $v \le 51$.
If there exists a skew-symmetric  $2$-$\{v;r,k;\lambda\}$
supplementary difference sets satisfying \eqref{eq:CS1}--\eqref{eq:CS3},
then it is equivalent to one of the supplementary difference sets
$(A,B)$ with $v \le 51$ in Table~\ref{Tab:SDS}.
\end{theorem}

For $v \ge 53$, 
due to the computational complexity,
our exhaustive computer search completed
a classification of skew-symmetric  $2$-$\{v;r,k;\lambda\}$
supplementary difference sets 
satisfying \eqref{eq:CS1}--\eqref{eq:CS3} for 
the following feasible parameters:
\begin{equation}\label{eq:par}
\begin{array}{cl}
(v,r,k,\lambda)=
&(55, 27,  0, 13),(55, 27,  1, 13),(59, 29,  0, 14),\\
&(59, 29,  1, 14),(61, 30,  6, 15),(61, 30, 10, 16),\\
&(63, 31,  0, 15),(63, 31,  1, 15),(67, 33,  0, 16),\\
&(67, 33,  1, 16),(71, 35,  0, 17),(71, 35,  1, 17),\\
&(75, 37,  0, 18),(75, 37,  1, 18).

\end{array}
\end{equation}
The skew-symmetric  $2$-$\{v;r,k;\lambda\}$
supplementary difference sets $(A,B)$ with parameters \eqref{eq:par}
are listed in  Table~\ref{Tab:SDS}.
For the feasible parameters $(v,r,k,\lambda)$ given in Table~\ref{Tab:Par},
the numbers $N(v,r,k,\lambda)$ of the inequivalent
skew-symmetric  $2$-$\{v;r,k;\lambda\}$
supplementary difference sets are also listed in the table.

\section{Classification of skew-symmetric circulant  $D$-optimal
designs meeting \eqref{eq:det}}\label{sec:ccsdd}

Skew-symmetric circulant $D$-optimal designs
$X(R_1,R_2)$ in \eqref{eq:D-opt}
are corresponding to  a certain class of
skew-symmetric supplementary difference sets 
satisfying \eqref{eq:CS1}--\eqref{eq:CS3}.
According to~\cite{KKNK}, we say that
circulant $D$-optimal designs meeting \eqref{eq:det} are {\em equivalent}
if the supplementary difference sets 
constructed by Lemma~\ref{lem:SDS-Doptimal} are equivalent.
In this section, as a consequence of the previous section, 
we give a classification of skew-symmetric circulant $D$-optimal
designs meeting \eqref{eq:det} for orders up to $110$.

Let $D$ be a circulant $D$-optimal design $X(R_1,R_2)$ in \eqref{eq:D-opt}
of order $n=2v$ meeting \eqref{eq:det}.
Here we suppose that 
$r$ and $k$ are the numbers of $-1$'s in the first rows of
$R_1$ and $R_2$, respectively.
If $D$ is skew-symmetric, then $r=\frac{v-1}{2}$ by Lemma~\ref{lem:skew}.

We call $(n,r,k)$ {\em feasible parameters} for 
skew-symmetric circulant $D$-optimal designs
if $(n,r,k)$ satisfies 
$r=\frac{v-1}{2}$ and \eqref{eq:1}. 
In Table~\ref{Tab:Dpara}, we list the feasible parameters
$(n,r,k)$ for $n \le 200$.

\begin{table}[thb]
\caption{Parameters of skew-symmetric circulant $D$-optimal designs}
\label{Tab:Dpara}
\begin{center}
{\small
\begin{tabular}{l|c} 
\noalign{\hrule height0.8pt}
\multicolumn{1}{c|}{$(n,r,k)$} & $N(n,r,k)$ \\
\hline
$(   6,   1,  0)$& 1 \\
$(  14,   3,  1)$& 1 \\
$(  26,   6,  3)$& 1 \\
$(  42,  10,  6)$& 1 \\
$(  62,  15, 10)$& 1 \\
$(  86,  21, 15)$& 0 \\
$( 114,  28, 21)$& ? \\
$( 146,  36, 28)$& ? \\
$( 182,  45, 36)$& ? \\
\noalign{\hrule height0.8pt}
\end{tabular}
}
\end{center}
\end{table}

Let $S_3$, $S_7$, $S_{13}$, $S_{21}$ and $S_{31}$ be
the skew-symmetric  $2$-$\{v;r,k;\lambda\}$
supplementary difference sets in Table~\ref{Tab:SDS}
with 
$(v,r,k,\lambda)=(3,1,0,0)$,
$(7,3,1,1)$,
$(13,6,3,3)$,
$(21,10,6,6)$ and 
$(31,15,10,10)$, respectively.
Let $D_{6}$, $D_{14}$, $D_{26}$, $D_{42}$ and $D_{62}$ be
the skew-symmetric circulant $D$-optimal designs $X(R_1,R_2)$ in \eqref{eq:D-opt}
of orders $6, 14, 26, 42$ and $62$
meeting \eqref{eq:det},
constructed by Lemma~\ref{lem:SDS-Doptimal} from
$S_3$, $S_7$, $S_{13}$, $S_{21}$ and $S_{31}$, respectively.
From the classification in Theorem~\ref{thm:SDS-Classification},
we have the following:

\begin{corollary}\label{cor:D}
Suppose that $n \le 110$.
If there exists a skew-symmetric 
circulant $D$-optimal design $X(R_1,R_2)$ in \eqref{eq:D-opt}
of order $n$ meeting \eqref{eq:det},
then it is equivalent to one of 
$D_6$, $D_{14}$, $D_{26}$, $D_{42}$ and $D_{62}$.
\end{corollary}

The numbers $N(n,r,k)$ of the inequivalent skew-symmetric circulant
$D$-optimal designs meeting \eqref{eq:det}
are also listed in Table~\ref{Tab:Dpara} for the 
feasible parameters $(n,r,k)$.

A classification of circulant $D$-optimal designs 
$X(R_1,R_2)$ in \eqref{eq:D-opt} meeting \eqref{eq:det}
was given in~\cite{KKNK} for orders $n \le 58$ and $n=66$, and
in~\cite{AH} for orders $n=62, 74$
(see~\cite{AH} for the revised classification for order $26$).
Our computer search found that $D_n$ $(n=6,14,26)$ is equivalent to
the circulant $D$-optimal design,
which is constructed by Lemma~\ref{lem:SDS-Doptimal} from 
the first supplementary difference set given 
in~\cite[Table 1]{KKNK},
$D_{42}$ is equivalent to
the circulant $D$-optimal design,
which is constructed by Lemma~\ref{lem:SDS-Doptimal} from 
the 19th supplementary difference set given 
in~\cite[Table 1]{KKNK}, and
$D_{62}$ is equivalent to
the circulant $D$-optimal design, 
which is constructed by Lemma~\ref{lem:SDS-Doptimal} from 
the 50th supplementary difference set given 
in~\cite[Appendix]{AH}.

\bigskip
\noindent {\bf Acknowledgment.}
The authors would like to thank the anonymous referees for 
helpful comments.
This work is supported by JSPS KAKENHI Grant Numbers 15K04976, 26610032.

\begin{landscape}
\begin{table}[thbp]
\caption{Skew-symmetric supplementary difference sets $(A,B)$}
\label{Tab:SDS}
\begin{center}
{\footnotesize
\begin{tabular}{l} 
\noalign{\hrule height0.8pt}
$(v,r,k,\lambda)=(3,1,0,0)$ \\
$A=\{2\}$
$B=\emptyset$\\
\hline
$(v,r,k,\lambda)=(7,3,0,1),(7,3,1,1)$\\
$A=\{3, 5, 6\}$
$B=\emptyset,\{0\}$\\
\hline
$(v,r,k,\lambda)=(11,5,0,2),(11,5,1,2)$\\
$A=\{2, 6, 7, 8, 10\}$
$B=\emptyset,\{0\}$\\
\hline
$(v,r,k,\lambda)=(13,6,3,3)$\\
$A=\{4, 7, 8, 10, 11, 12\}$
$B=\{0, 2, 8\}$\\
\hline
$(v,r,k,\lambda)=(19,9,0,4),(19,9,1,4)$\\
$A=\{2, 3, 8, 10, 12, 13, 14, 15, 18\}$
$B=\emptyset,\{0\}$\\
\hline
$(v,r,k,\lambda)=(21,10,6,6)$\\
$A=\{2, 3, 9, 11, 13, 14, 15, 16, 17, 20\}$
$B=\{0, 1, 7, 9, 12, 17\}$\\
\hline
$(v,r,k,\lambda)=(23,11,0,5),(23,11,1,5)$\\
$A=\{5, 7, 10, 11, 14, 15, 17, 19, 20, 21, 22\}$
$B=\emptyset,\{0\}$\\
\hline
$(v,r,k,\lambda)=(29,14,7,8)$\\
$A=\{4, 5, 6, 8, 9, 10, 12, 13, 15, 18, 22, 26, 27, 28\}$
$B=\{0, 1, 11, 13, 15, 18, 21\}$\\
\hline
$(v,r,k,\lambda)=(31,15,0,7),(31,15,1,7)$\\
$A=\{3, 6, 11, 12, 13, 15, 17, 21, 22, 23, 24, 26, 27, 29, 30\}$
$B=\emptyset,\{0\}$\\
\hline
$(v,r,k,\lambda)=(31,15,6,8)$\\
$A=\{3, 6, 11, 12, 13, 15, 17, 21, 22, 23, 24, 26, 27, 29, 30\}$
$B=\{0, 1, 15, 20, 22, 28\}$\\
\hline
$(v,r,k,\lambda)=(31,15,10,10)$\\
$A=\{4, 6, 7, 12, 16, 17, 18, 20, 21, 22, 23, 26, 28, 29, 30\}$
$B=\{0, 1, 4, 5, 8, 11, 16, 18, 20, 29\}$\\
\hline
$(v,r,k,\lambda)=(43,21,0,10),(43,21,1,10)$\\
$A=\{2, 3, 5, 7, 8, 12, 18, 19, 20, 22, 26, 27, 28, 29, 30, 32, 33, 34, 37, 39, 42\}$
$B=\emptyset,\{0\}$\\
\hline
$(v,r,k,\lambda)=(47,23,0,11),(47,23,1,11)$\\
$A=\{5, 10, 11, 13, 15, 19, 20, 22, 23, 26, 29, 30, 31, 33, 35, 38, 39, 40, 41, 43, 44, 45, 46\}$
$B=\emptyset,\{0\}$\\
\hline
$(v,r,k,\lambda)=(59,29,0,14),(59,29,1,14)$\\
$A=\{2, 6, 8, 10, 11, 13, 14, 18, 23, 24, 30, 31, 32, 33, 34, 37, 38, 39, 40, 42, 43, 44, 47, 50, 52, 54, 55, 56, 58\}$
$B=\emptyset,\{0\}$\\
\hline
$(v,r,k,\lambda)=(67,33,0,16),(67,33,1,16)$\\
$A=\{2, 3, 5, 7, 8, 11, 12, 13, 18, 20, 27, 28, 30, 31, 32, 34, 38, 41, 42, 43, 44, 45, 46, 48, 50, 51, 52, 53, 57, 58, 61, 63, 66\}$
$B=\emptyset,\{0\}$\\
\hline
$(v,r,k,\lambda)=(71,35,0,17),(71,35,1,17)$\\
$A=\{7, 11, 13, 14, 17, 21, 22, 23, 26, 28, 31, 33, 34, 35, 39, 41, 42, 44, 46, 47, 51, 52, 53, 55, 56, 59, 61, 62, 63, 65, 66, 67, 68, 69, 70\}$
$B=\emptyset, \{0\}$\\
\hline
\noalign{\hrule height0.8pt}
\end{tabular}
}
\end{center}
\end{table}
\end{landscape}

\end{document}